\documentclass[11pt]{amsart}

\usepackage{amssymb}
\usepackage[all]{xy}
\usepackage{color}

\newtheorem{thm}{Theorem}[section]
\newtheorem{lem}[thm]{Lemma}
\newtheorem{cor}[thm]{Corollary}
\newtheorem{prop}[thm]{Proposition}

\newtheorem{assum}[thm]{Assumption}

\theoremstyle{definition}

\newtheorem{defn}[thm]{Definition}
\newtheorem{defns}[thm]{Definitions}

\newtheorem{notation}[thm]{Notation}
\newtheorem{ex}[thm]{Example}

\theoremstyle{remark}

\newtheorem{rem}[thm]{Remark}

\numberwithin{equation}{section}

\newcommand{\thmref}[1]{Theorem~\ref{#1}}
\newcommand{\corref}[1]{Corollary~\ref{#1}}
\newcommand{\secref}[1]{\S\ref{#1}}

\newcommand{\propref}[1]{Proposition~\ref{#1}}
\newcommand{\lemref}[1]{Lemma~\ref{#1}}

\newcommand{\Mod}{\mathsf{Mod}}
\newcommand{\Alg}{\mathsf{Alg}}
\newcommand{\AlgO}{\mathsf{Alg}_{\Oo}}
\newcommand{\Sym}{\mathsf{Sym}}
\newcommand{\biSym}{\mathsf{biSym}}

\newcommand{\modd}{{\text{-mod}}}

\newcommand{\co}{\circ_{\Oo}}

\newcommand{\U}{{\mathcal  O}}

\newcommand{\Oo}{{\mathcal  O}}

\newcommand{\sm}{\wedge}

\newcommand{\ra}{\rightarrow}
\newcommand{\xra}{\xrightarrow}
\newcommand{\la}{\leftarrow}
\newcommand{\xla}{\xleftarrow}

\newcommand{\hra}{\hookrightarrow}

\begin{document}

\title[Composition products]{Operad bimodules, and composition products on Andr\'e--Quillen filtrations of algebras}

\author[Kuhn]{Nicholas J.~Kuhn}

\address{Department of Mathematics \\ University of Virginia \\ Charlottesville, VA 22904}

\email{njk4x@virginia.edu}


\author[Pereira]{Lu\'is A.~Pereira}
\address{Department of Mathematics \\ University of Virginia \\ Charlottesville, VA 22904}
\email{lp2h@virginia.edu}
\thanks{}

\date{February 21, 2016.}


\begin{abstract}  If $\Oo$ is a reduced operad in a symmetric monoidal category of spectra ($S$--modules), an $\Oo$--algebra $I$ can be viewed as analogous to the augmentation ideal of an augmented algebra.  Implicit in the literature on Topological Andr\'e--Quillen homology is that such an $I$ admits a canonical (and homotopically meaningful) decreasing $\Oo$--algebra filtration $I \xla{\sim} I^1 \la I^2 \la I^3 \la \dots$ satisfying various nice properties analogous to powers of an ideal in a ring.

In this paper, we are explicit about these constructions. With $R$ a commutative $S$--algebra, we use the bar construction as a derived version of functors of the form $I \mapsto M\co I$, where $M$ is an $\Oo$--bimodule, and $I$ is an $\Oo$--algebra in $R$--modules.   Letting $M$ run through a decreasing $\Oo$--bimodule filtration of $\Oo$ itself then yields the augmentation ideal filtration as above.  We then note that the composition structure of the operad induces products $(I^i)^j \ra I^{ij}$, fitting nicely with previously studied structure.

As a formal consequence, an $\Oo$--algebra map $I \ra J^d$ induces compatible maps $I^n \ra J^{dn}$ for all $n$.  This is an essential tool in the first author's study of Hurewicz maps for infinite loop spaces, and its utility is illustrated here with a lifting theorem.

\end{abstract}

\maketitle

\section{Introduction} \label{INTRO}

Let $S\modd$ be the category of symmetric spectra \cite{hss}, one of the standard symmetric monoidal models for the category of spectra.  Let $S$ denote the sphere spectrum, and let $\Oo$ be a reduced operad in $S\modd$.  If $R$ is a commutative $S$--algebra, we let $\AlgO(R)$ denote the category of $\Oo$--algebras in $R$--modules.

The starting point of this paper is the observation that, if $M$ is a reduced $\Oo$--bimodule, and $I\in \AlgO(R)$, then $M \co I$ is again in $\AlgO(R)$, and that many interesting constructions on $\Oo$--algebras are derived versions of functors of $I$ of this form.

Our first goal here is to present the basic properties of a suitable derived version of this construction, the bar construction $B(M,\Oo,I)$, noting how structure on the category of $\Oo$--bimodules is reflected in the category of endofunctors of $\Oo$--algebras.  Perhaps the least familiar of these is that a levelwise homotopy cofibration sequence in the bimodule variable $M$ induces a homotopy fibration sequence in $\AlgO(R)$: see \thmref{NAT IN M THM}(b).

We then apply these general results to recover constructions on $\Oo$--algebras implicitly studied earlier by various people (notably \cite{HH}), but add new structure.

An $\Oo$--algebra $I$ can be viewed as similar to the augmentation ideal in an augmented ring.  Applying our bar construction to a natural decreasing $\Oo$--bimodule filtration of $\Oo$ itself, shows that $I \in \AlgO(R)$ admits a homotopically meaningful natural `augmentation ideal filtration':
$$ I \xla{\sim} I^1 \la I^2 \la I^3 \la \dots$$
with `$I^n/I^{n+1}$' determined by $\Oo(n)$ and its topological Andr\'e-Quillen $R$--module $TQ(I)$. $TQ(I)$ can be informally viewed as `$I/I^2$': its study goes back to \cite{basterra}.  Our model makes it easy to analyze connectivity properties: if $R$ and $\Oo$ are connective and $I$ is $(c-1)$ connected, then $I^n$ will be $(nc-1)$--connected.

We now take advantage of the observation that a pairing of bimodules
$$L \co M \ra N$$
will induce a natural transformation of functors of $I$
$$ L \co (M \co I) \ra N \co I,$$
and similarly on our derived model.
Applied to pairings among our $\Oo$--bimodule filtration of $\Oo$ induced by its operad structure, we obtain natural pairings
$$ (I^n)^m \ra I^{mn}$$
satisfying expected properties.  As a formal consequence, an $\Oo$--algebra map $I \ra J^d$ induces compatible maps
$I^n \ra J^{dn}$ for all $n$.

This seems to be fundamental structure which has not previously appeared in the literature.
The next result is a consequence illustrating its utility.
\begin{thm}  Let $f:I \ra J$ be a map in the homotopy category $ho \AlgO(R)$.  If $f$ factors as $f = f_s \circ \dots \circ f_1$ such that $TQ(f_i)$ is null for all $i$, then there is a lifting in $ho \AlgO(R)$:
\begin{equation*}
\xymatrix{
 & J^{2^s} \ar[d]  \\
I \ar[r]^f \ar@{.>}[ru]^{\tilde f}& J. }
\end{equation*}
\end{thm}
\noindent We restate this, with slightly different notation, as \thmref{AQ lifting thm}.

Further applications in this spirit can be seen in work by the first author on Hurewicz maps of infinite loopspaces \cite{kuhn hurewicz}, the project whose needs motivated this paper.  Also critical in \cite{kuhn hurewicz}, is that we verify that there are sensible `change-of-rings' formulae for our constructions as $R$ varies.

The paper is organized as follows.

In \secref{GEN RESULTS SEC}, we first introduce the setting in which we wish to work.  This includes a well chosen, and slightly delicate, model structure on $\AlgO(R)$, which piggybacks off of the `positive' model structure on $S$--mod first exploited in \cite{shipley}, and is in the spirit of \cite{H07}.  We then state the basic homotopy properties of our derived version of $M \co I$.  To summarize: what we find most compelling is that on one hand, our constructions connect nicely to $TQ(I)$, and on the other, they are well suited to iteration using the monoidal properties of $\circ$.

In \secref{comp sec}, we apply the result of the previous section to the augmentation ideal filtration of an $\Oo$--algebra, and deduce lifting results as above.

The deeper proofs from \secref{GEN RESULTS SEC} are deferred to \secref{deferred proofs sec}, which itself is supported by Appendix \ref{app a}.  Much of the technical work consists of generalizing results and arguments from \cite{Pereira2014} from $S$--mod to $R$--mod for a general $R$.

\section{General results about derived composition products} \label{GEN RESULTS SEC}

\subsection{Our categories of modules and algebras} \label{OUR SETTING}

In this paper, the category of $S$--modules will mean the category of symmetric spectra as defined in \cite{hss}: here we recall that $X \in S\modd$ consists of a sequence $X_0,X_1,X_2, \dots$ of simplicial sets equipped with extra structure.

With the smash product as product and sphere spectrum $S$ as unit, $S\modd$ is a closed symmetric monoidal category. There is a notion of weak equivalence, and various model structures on $S\modd$ compatible with these, such that the resulting quotient category models the standard stable homotopy category.

Recall that a symmetric sequence in $S\modd$ then consists of a sequence
$$X(0),X(1),X(2),\dots,$$ where $X(n)$ is a symmetric spectrum equipped with an action of the $n$th symmetric group $\Sigma_n$.

The category of such symmetric sequences in $S\modd$, $\Sym(S)$, admits a composition product $\circ$ defined by
\begin{equation}\label{COMPPROD EQ}
 (X \circ Y)(s) = \bigvee_r X(r) \sm_{\Sigma_r} \left(\bigvee_{\phi: \bf s \ra \bf r} Y(s_1(\phi)) \sm \dots \sm Y(s_r(\phi))\right),
\end{equation}
where ${\bf s} = \{1,\dots,s\}$ and $s_k(\phi)$ is the cardinality of $\phi^{-1}(k)$.
With this product, $(\Sym(S), \circ, S(1))$ is monoidal, where $S(1) = (*,S,*,*, \dots)$.

An operad $\Oo$ is then a monoid in this category, and one makes sense of left $\Oo$--modules, right $\Oo$--modules, and $\Oo$--bimodules in the usual way.  Furthermore, if $X$ is a right $\Oo$--module, and $Y$ is a left $\Oo$--module, the symmetric sequence $X \co Y$ can be defined as the coequalizer in $\Sym(S)$ of the two evident maps
$$ X \circ \Oo \circ Y \begin{array}{c} \longrightarrow \\[-.1in] \longrightarrow
\end{array} X \circ Y.$$
Extra structure on $X$ or $Y$ then can induce evident extra structure on $X \co Y$.

For the purposes of this paper, it is natural to require that our operads $\Oo$ and bimodules $M$ be reduced: $\Oo(0)=*=M(0)$.  By contrast, an $\Oo$--algebra is a left $\Oo$--module $I$ concentrated in level 0: $I(n)=*$ for all $n>0$.

If $R$ is a commutative $S$--algebra, these definitions and constructions extend to the category of $R$--modules.  Furthermore, one can mix and match: for example, if $X$ is a symmetric sequence in $S\modd$ and $Y$ is a symmetric sequence in $R\modd$, $X \circ Y$ will be the symmetric sequence in $R\modd$ with
$$ (X \circ Y)(s) = \bigvee_r X(r) \sm_{\Sigma_r} \left(\bigvee_{\phi: \bf s \ra \bf r} Y(s_1(\phi)) \sm_R \dots \sm_R Y(s_r(\phi))\right).$$

We denote by $\Sym(R)$ the category of symmetric sequences in $R\modd$,  $\AlgO(R)$ the category of $\Oo$-algebras in $R\modd$ and $\Mod^l(R)$ the category of left $\Oo$--modules in $\Sym(R)$.

\subsection{Model structures}\label{MODELSTRUCTURE SEC}

We specify model structures on the various categories just described.

We accept as given the $S$--model structure on symmetric spectra (called $S$--modules in this paper) as defined in \cite[Defn.5.3.6]{hss} and \cite[Thm.2.4]{shipley}.  This structure is monoidal with respect to the smash product \cite[Cor.5.3.8]{hss}.

We then give $\Sym(S)$ its associated {\em injective} model structure: weak equivalences and cofibrations are those morphisms which are levelwise weak equivalences and cofibrations in $S\modd$.  That this structure exists was checked in \cite{Pereira2014}: in that reference, see Theorem 3.8 and \S5.3 \footnote{This structure is different from the associated projective structure used in \cite{H07,H09,HH}.}.

As in \cite[\S 15]{mmss}, \cite{shipley}, \cite{HH}, and \cite[\S 5.3]{Pereira2014}, we need `positive' variants of these model structures. Weak equivalences will be as before, but there are fewer cofibrations: for $X \ra Y$ in $S\modd$ to be a positive cofibration, we now insist that $X_0 \ra Y_0$ also be an isomorphism, and for $M \ra N$ in $\Sym(S)$ to be a positive cofibration, we now insist that $M(0)_0 \ra N(0)_0$ also be an isomorphism\footnote{On $\Sym(S)$, this agrees with \cite{Pereira2014} but is different from \cite{HH}, where it is required that $M(n)_0 \ra N(n)_0$ be an isomorphism for all $n$.}.  It is worth noting that if $M \in \Sym(S)$ is reduced, then it is positive cofibrant exactly when each $M(n)$ is cofibrant, when viewed in $S\modd$.

Given a commutative $S$--algebra $R$, the positive $R$--model structure on $R$--modules is then defined to be the projective structure induced from that on $S\modd$ with its positive structure: weak equivalences and fibrations in $R\modd$ are the maps which are weak equivalences and positive fibrations in $S\modd$.   Similarly, we define the positive structure on $\Sym(R)$, the category of symmetric sequences in $R\modd$, to be the projective structure induced from that on $\Sym(S)$ with its positive structure: weak equivalences and fibrations in $\Sym(R)$ are the maps which are weak equivalences and positive fibrations in $\Sym(S)$.

The following theorem is an immediate consequence of \cite[Thm. 1.4]{Pereira2014}  (see also \cite{H07}), and special cases go back to \cite{shipley}.

\begin{thm}\label{PROJPOS THM}
  $\Alg_{\Oo}(R)$ has a projective model structure induced from the positive structure on $R\modd$: $f: I \ra J$ is a weak equivalence if it is one in $R\modd$ (and thus in $S\modd$), and a fibration if it is a positive fibration in $R\modd$ (and thus in $S\modd$). Similarly, $\Mod^l_{\Oo}(R)$ has a projective model structure induced from the positive structure on $\Sym(R)$: $f: M \ra N$ is a weak equivalence if it is one in $\Sym(R)$ (and thus in $\Sym(S)$), and a fibration if it is a positive fibration in $\Sym(R)$ (and thus in $\Sym(S)$).
\end{thm}

The next lemma says that the model structure on $\Alg_{\U}(R)$ is really the same as the model structure on $\Mod^l_{\Oo}(R)$, restricted to the subcategory of modules concentrated in degree 0.

\begin{lem}  If $I \ra J$ is a cofibration in $\Alg_{\Oo}(R)$, then it is a cofibration in $\Mod^l_{\Oo}(R)$, when $I$ and $J$ are regarded as objects in $\Sym(R)$ concentrated in level 0.
\end{lem}
\begin{proof}  The inclusion $\Alg_{\Oo}(R) \hra \Mod^l_{\Oo}(R)$ has right adjoint given by $M \mapsto M(0)$.  This is a Quillen pair, as it is easily checked that this right adjoint preserves weak equivalences and fibrations.
\end{proof}

\subsection{Cofibrancy assumption on $\Oo$ and first consequences}

Unless stated otherwise, we make the following standing cofibrancy assumption about our operad $\Oo$.

\begin{assum}  \label{COFIB ASSUM} The map $S(1) \ra \Oo$ is assumed to be a positive cofibration in $\Sym(S)$.
\end{assum}

As $\Oo(0)=*$ has been assumed earlier, equivalently this means that, in $S\modd$, $S\ra \Oo(1)$ is a cofibration, and $\Oo(n)$ is cofibrant for all $n$.

\begin{notation} Let $\AlgO(R)^{c}$ be the full subcategory of $\AlgO(R)$ consisting of $\Oo$--algebras in $R\modd$ which are cofibrant when just viewed as $R$--modules.
\end{notation}

A key advantage of our particular model structure on $\AlgO(R)$ is that the following property holds.

\begin{prop} \label{COFFGT PROP} The forgetful functor $\AlgO(R) \ra R\modd$ preserves cofibrations between cofibrant objects.  In particular, if $I$ is cofibrant in $\AlgO(R)$, then $I \in \AlgO(R)^{c}$.
\end{prop}

When $R=S$ this is \cite[Theorem 1.5]{Pereira2014}. We defer the proof of the general case to \secref{deferred proofs sec}.

It follows that a functorial cofibrant replacement functor on $\AlgO(R)$
takes values in $\AlgO(R)^c$.

More elementary, but also useful is that $\AlgO(R)^{c}$ is well behaved under change of rings.

\begin{lem} Let $R \ra R^{\prime}$ be a map of commutative $S$--algebras. Then
$$R'\wedge_R \underline{\text{\hspace{.13in}}}: \AlgO(R) \to \AlgO(R^{\prime})$$
 restricts to a functor
$$ R'\wedge_R \underline{\text{\hspace{.13in}}}: \AlgO(R)^{c} \to \AlgO(R^{\prime})^c$$
which preserves weak equivalences.
\end{lem}

\begin{proof}
This is immediate since $R'\wedge_R \underline{\text{\hspace{.13in}}}$ is left adjoint to a forgeful functor that is easily seen to be right Quillen.
\end{proof}

\subsection{General properties of the bar construction} \label{general results}

We will make much use of the bar construction.  Given an $\Oo$--bimodule $M$ and $I \in \AlgO(R)$, $B(M,\Oo,I) \in \AlgO(R)$ is defined as the geometric realization of the simplicial object $B_{\bullet}(M,\Oo,I)$ in $R\modd$ defined by
$$ B_n(M,\Oo,I) =  M \circ \overbrace{\Oo \circ \cdots \circ \Oo}^n \circ I.$$

Similarly if $M$ and $N$ are $\Oo$--bimodules, then $B(M,\Oo,N)$ is again an $\Oo$-bimodule.

The theme of the next set of results is that this construction is well behaved when the $\Oo$--bimodules are positive cofibrant in $\Sym(S)$, and $I \in \AlgO(R)$ is cofibrant in $R\modd$.  (We recall that a reduced $M \in \Sym(S)$ is positively cofibrant exactly when it is levelwise cofibrant.)

\begin{prop}\label{BARCOF PROP}
Let $M$ and $N$ be levelwise cofibrant $\Oo$--bimodules.  Then $B(M\,\Oo,N)$ is again levelwise cofibrant.  Similarly, if $M$ is levelwise cofibrant and $I \in \AlgO(R)^c$, then $B(M,\Oo,I) \in \AlgO(R)^c$.
\end{prop}

The first statement is immediately implied by \cite[Theorem 1.6]{Pereira2014} which says that $B_{\bullet}(M,\Oo,N)$ is Reedy cofibrant in the category of simplicial objects in $\Sym(S)$.  We defer the proof of the second statement for general $R$ to \secref{deferred proofs sec}.

We also record the following, which shows that the bar construction can be usefully used as a derived circle product.

\begin{prop}\label{BARDEREQ PROP} Let $M$ be a levelwise cofibrant right $\Oo$--module.
If $I$ is cofibrant in $\AlgO(R)$, the natural map $B(M,\Oo,I) \ra M \co I$ is a weak equivalence.  Similarly if $N$ is cofibrant in $\Mod^l_{\Oo}(S)$, then $B(M,\Oo,N) \ra M \co N$ is a weak equivalence.
\end{prop}

This will also be proved in \secref{deferred proofs sec}.

To emphasize the functors defined by levelwise cofibrant bimodules, we change notation.

\begin{defn}  If $M$ is a levelwise cofibrant $\Oo$--bimodule, define
$$ F^R_M: \AlgO(R)^c \ra \AlgO(R)^c$$
by the formula $F^R_M(I) = B(M,\Oo,I)$.
\end{defn}

\begin{thm} \label{NAT IN M THM}

The $F^R_{M}$ construction satisfies the following properties.

\noindent{\bf (a)} \ $(M,I)\mapsto F^R_{M}(I)$ takes weak equivalences in either the $M$ or $I$ variable to weak equivalences in $\AlgO(R)$.

\noindent{\bf (b)} \ A levelwise homotopy cofibration sequence of levelwise cofibrant $\Oo$--bimodules
     $$ L \ra M \ra N$$
     induces a homotopy fibration sequence in $\AlgO(R)$
     $$ F^R_{L}(I) \ra F^R_{M}(I) \ra F^R_{N}(I).$$

\noindent{\bf (c)} \ There is a natural isomorphism of functors:
$$F^R_M\circ F^R_N \simeq F^R_{B(M,\Oo, N)}.$$

\noindent{\bf (d)} \ Let $R \ra R^{\prime}$ be a map of commutative $S$--algebras.  There is a natural isomorphism in $\AlgO(R^{\prime})$:
$$F^{R^{\prime}}_M(R^{\prime} \sm_R I) \simeq R^{\prime} \sm_R F^R_M(I).$$
\end{thm}

Parts (a) and (b) will be proved in \secref{deferred proofs sec}.  By contrast, parts (c) and (d) are straightforward.  Part (c) follows from the natural isomorphism
$$ B(M,\Oo,B(N,\Oo,I)) \simeq B(B(M,\Oo,N),\Oo,I),$$
while part (d) follows from the natural isomorphism
$$ R^{\prime} \sm_R B(M,\Oo,I) \simeq B(M,\Oo,R^{\prime} \sm_R I).$$

\begin{rem}  As there is a natural map $B(M,\Oo,N) \ra M\co N$, it follows that a bimodule pairing
$$ \mu: M \co N \ra L$$
induces a natural transformation
$$ \mu: F^R_M \circ F^R_N \ra F^R_L$$
defined as the composite
$$F^R_M \circ F^R_N \simeq F^R_{B(M,\Oo, N)} \ra F^R_{M\co N} \ra F^R_L.$$
See \secref{comp sec} for examples of this.
\end{rem}

\subsection{Topological Andr\'e--Quillen homology} \label{TQ section}

In the next two subsections, we consider our constructions when $M$ is concentrated in just one level, i.e., there exists an $n$ such that $M(m) = *$ for all $m \neq n$.  We show that then $F^R_M(I)$ is determined by $M(n)$ together with the Topological Andr\'e--Quillen homology of $I$.

We first need to define this last last term in our context.  The $S$--module $\Oo(1)$ will be an associative $S$--algebra, and can be viewed as an operad concentrated in level 1.  From this point of view, the evident maps $\Oo(1) \ra \Oo$ and $\Oo \ra \Oo(1)$ are both maps of operads, and the second of these gives $\Oo(1)$ the structure of an $\Oo$--bimodule concentrated in level 1.

Let $R\Oo(1)\modd$ be the category of $R \sm \Oo(1)$--modules. It is illuminating to note that this category is also $\Alg_{\Oo(1)}(R)$, when one views $\Oo(1)$ as an operad.  The map $\Oo \ra \Oo(1)$ induces a functor
$$ z: R\Oo(1)\modd \ra \AlgO(R)$$
with left adjoint
$$ Q = \Oo(1) \co \underline{\text{\hspace{.13in}}}: \AlgO(R) \ra R\Oo(1)\modd$$
making the pair of functors into a Quillen pair.

\begin{defn}  Define $ TQ: \AlgO(R)^c \ra R\Oo(1)\modd$ by the formula
$TQ(I) = B(\Oo(1),\Oo,I)$.
\end{defn}

The next proposition is a special case of \propref{BARDEREQ PROP}.

\begin{prop}  If $I$ is cofibrant in $\AlgO(R)$, the natural map $TQ(I) \ra Q(I)$ is an equivalence.
\end{prop}

As $TQ$ is thus equivalent to the left derived functor of the left Quillen functor $Q$, one has the next two consequences.

To state the first, we let $[I,J]_{\Alg}$ denote morphisms between $I$ and $J$ in the homotopy category of $\AlgO(R)$, and we similarly let $[M,N]_{\Mod}$ denote morphisms between $M$ and $N$ in the homotopy category of $R\Oo(1)\modd$.

\begin{cor} \label{TQ homotopy cor} There is an adjunction in the associated homotopy categories:
$$ [TQ(I), N]_{\Mod} \simeq [I,z(N)]_{\Alg}.$$
\end{cor}

\begin{cor} \label{linear cor} If $I \ra J \ra K$ is a homotopy cofibration sequence in $\AlgO(R)$, then
$$TQ(I) \ra TQ(J)\ra TQ(K)$$
is a homotopy cofibration sequence in $R\Oo(1)\modd$.
\end{cor}

The next result is a particular instance of \thmref{NAT IN M THM}(d).

\begin{prop} \label{TQ change of R prop} Let $R \ra R^{\prime}$ be a map of commutative $S$--algebras.  There is a natural isomorphism
$$TQ(R^{\prime} \sm_R I) \simeq R^{\prime} \sm_R TQ(I).$$
\end{prop}
\noindent The first `$TQ$' here is with respect to the $S$--algebra $R^{\prime}$.

\subsection{$\Oo$--bimodules with one term} \label{one term  section}  Again we view $\Oo(1)$ as an operad concentrated in level 1.

Suppose $M \in \Sym(S)$ is a right $\Oo(1)$--module, i.e., one has $M \circ \Oo(1) \ra M$ making appropriate diagrams commute.  Unraveling the definitions, one sees that this structure map amounts to $\Sigma_n$--equivariant maps
$$ M(n) \sm \Oo(1)^{\sm n} \ra M(n)$$
exhibiting $M(n)$ as an $\Oo(1)^{\sm n}$--module.  Equivalently, each $M(n)$ will be a right $\Sigma_n\wr \Oo(1)$--module, where $\Sigma_n\wr \Oo(1)$ is the associative algebra with underlying $S$--module $\bigvee_{\sigma \in \Sigma_n} \Oo(1)^{\sm n}$, and evident `twisted' multiplication.

From this, it is easy to see that if $J \in \Alg_{\Oo(1)}(R) = R\Oo(1)\modd$, then
$$ M \circ_{\Oo(1)}J = \bigvee_n M(n) \sm_{\Sigma_n\wr \Oo(1)}J^{\sm_R n}.$$

Now suppose, given `$M(n)$', a left $\Oo(1)$--module that is also a right $\Sigma_n\wr \Oo(1)$--module.  Abusing notation, we will also write $M(n)$ for the symmetric sequence concentrated at level $n$:
$$M(n) = (*, \dots, *, M(n), *, \dots).$$
From this point of view, $M(n)$ is precisely an $\Oo(1)$--bimodule, where $\Oo(1)$ is viewed as an operad.  Furthermore, an $\Oo$--bimodule structure on $M(n)$ will necessarily be an $\Oo(1)$--bimodule structure pulled back via the map of operads $\Oo \ra \Oo(1)$.

\begin{thm} \label{ONE TERM THM}  Suppose $M(n)$ is also a cofibrant $S$--module.   Then, for $I \in \AlgO(R)^c$, there is a natural isomorphism
$$F^R_{M(n)}(I) = z(M(n) \sm_{\Sigma_n \wr \Oo(1)} TQ(I)^{\sm_R n}),$$
and a natural equivalence
$$z(B(M(n), \Oo(1), TQ(I))) \xra{\sim} F^R_{M(n)}(I).$$
\end{thm}

\begin{proof}  We repress some applications of $z$, the pullback along $\Oo \ra \Oo(1)$.

Firstly, one has natural isomorphisms
\begin{equation*}
\begin{split}
M(n) \sm_{\Sigma_n \wr \Oo(1)} TQ(I)^{\sm_R n} &
= M(n) \sm_{\Sigma_n \wr \Oo(1)} B(\Oo(1),\Oo,I)^{\sm_R n} \\
  & = M(n) \circ_{\Oo(1)} B(\Oo(1),\Oo,I) \\
  & = B(M(n),\Oo,I) \\
  & = F^R_{M(n)}(I).
\end{split}
\end{equation*}

Secondly, the equivalence $B(M(n),\Oo(1),\Oo(1)) \xra{\sim} M(n)$ induces the equivalence:
\begin{equation*}
\begin{split}
B(M(n), \Oo(1), TQ(I)) &
= B(M(n), \Oo(1), B(\Oo(1),\Oo,I)) \\
  & = B(B(M(n), \Oo(1),\Oo(1)),\Oo,I) \\
  & \xra{\sim} B(M(n),\Oo,I) \\
  & = F^R_{M(n)}(I).
\end{split}
\end{equation*}

\end{proof}

\begin{cor}  \label{one term cor} Let $f: I \ra J$ be a morphism in $\AlgO(R)^c$. With $M(n)$ as in the theorem, if $TQ(f)$ is a weak equivalences, so is $F^R_{M(n)}(f)$.
\end{cor}

\subsection{The Goodwillie tower of $F^R_M$} \label{goodwillie tower sec}

The second author has studied Goodwillie calculus on the category $\AlgO(R)$ \cite{Pereira2013}. Here we sketch how our results above lead to an understanding of the Goodwillie tower of the functor $F^R_M$.

Given a levelwise cofibrant $\Oo$--bimodule $M$, let $M^{\leq n}$ denote the  $\Oo$--bimodule with
\begin{equation*}
M^{\leq n}(k) =
\begin{cases}
M(k) & \text{if } k\leq n \\ * & \text{if } k > n.
\end{cases}
\end{equation*}

\begin{defn}  Let $P_nF^R_M = F^R_{M^ {\leq n}}: \AlgO(R)^c \ra \AlgO(R)^c$.
\end{defn}

\begin{thm}  The Goodwillie tower of the functor $F^R_M$ identifies with
$$ P_1F^R_M \la P_2F^R_M \la P_3F^R_M \la \dots,$$
and its $n$th derivative $\partial_n F^R_M$ identifies with $M(n)$.
\end{thm}
\begin{proof}[Sketch proof]  The sequence of $\Oo$-bimodules
$$ M(n) \ra M^{\leq n} \ra M^{\leq (n-1)}$$
satisfies the hypothesis of \thmref{NAT IN M THM}(b).  Thus the homotopy fiber of the map
$$ P_nF^R_M(I) \ra P_{n-1}F^R_M(I)$$
identifies as $F^R_{M(n)}(I)$, which \thmref{ONE TERM THM} tells us is
$$z(M(n) \sm_{\Sigma_n \wr \Oo(1)} TQ(I)^{\sm_R n}).$$
This is a homogeneous $n$--excisive functor: note that \corref{linear cor} first tells us that $TQ$ is a homogeneous linear functor.  (See \cite[Theorem 3.2]{Pereira2013} for more detail.)

It follows that $P_nF^R_M$ is $n$--excisive.  With a bit more care, one can now check that the natural transformation $F^R_M \ra P_nF^R_M$ identifies with the map from $F^R_{M}$ to its $n$--excisive quotient: the proof of \cite[Theorem 4.3]{Pereira2013} generalizes immediately to our setting.
\end{proof}

Under connectivity hypotheses, one gets very concrete convergence estimates.  Say that $X \in \Sym(S)$ is connective if each $X(n) \in S\modd$ is connective, i.e. $-1$--connected.

\begin{prop} \label{connectivity prop}  If $R$, $M$, and $\Oo$ are connective, and $I$ is $(c-1)$--connected, then the map $F^R_M(I) \ra P_nF^R_M(I)$ is $(n+1)c$--connected.
\end{prop}
\begin{proof} We need to show that the homotopy fiber is $((n+1)c-1)$--connected. By \thmref{NAT IN M THM}(b), this homotopy fiber identifies with $B(M^{>n}, \Oo,I)$ where \begin{equation*}
M^{> n}(k) =
\begin{cases}
M(k) & \text{if } k> n \\ * & \text{if } k \leq  n.
\end{cases}
\end{equation*}
This fiber then is the homotopy colimit (in $R$--modules) of a diagram of $R$--modules of the form
$$ M(r) \sm \Oo(s_1) \sm  \dots \sm \Oo(s_k) \sm I^{\sm_R t},$$
with $t\geq r > n$.  In particular, it is a homotopy colimit of a diagram of $((n+1)c-1)$--connected $R$--modules, and so is itself $((n+1)c-1)$--connected.
\end{proof}

These results also show the following, when combined with \corref{one term cor}.

\begin{thm} Let $f: I \ra J$ be a morphism in $\AlgO(R)^c$. If $TQ(f)$ is a weak equivalence, so is $P_nF^R_M(f)$ for any $n$ and any levelwise cofibrant $\Oo$--bimodule $M$.  Furthermore, if $R$, $M$, and $\Oo$ are connective, and $I$ and $J$ are 0--connected, then $F^R_M(f)$ is itself will be a weak equivalence.
\end{thm}

Special cases of this theorem appear in \cite{kuhn-TAQtowers} and \cite{HH}.

\section{Application to the augmentation ideal filtration} \label{comp sec}

In our constructions, when the $\Oo$--bimodule is $\Oo$ itself, the resulting functor $I \mapsto F^R_{\Oo}(I)= B(\Oo,\Oo,I)$ is naturally weakly equivalent to the identity.  In this section we study structure on the `augmentation ideal' filtration of $I$ arising from using the levelwise bimodule filtration of $\Oo$ in conjunction with the operad structure $\Oo \circ \Oo \ra \Oo$.

\subsection{Construction and basic properties of the filtration}
\begin{defns} Let $1 \leq i < m \leq \infty$.

\noindent{\bf (a)} Let $\Oo_i^m$ denote the $\Oo$--bimodule with
$
\Oo_i^m(k) =
\begin{cases}
\Oo(k) & \text{if } i \leq k < m \\ * & \text{ otherwise}.
\end{cases}
$

\vskip 8pt

\noindent{\bf (b)} For $I \in \AlgO(R)^c$, let $I^i_m = F^R_{\Oo^m_i}(I) = B(\Oo^m_i, \Oo, I)$.
\end{defns}

Note that there is a natural weak equivalence $I^1_{\infty} \ra I$. We sometimes write $I^i$ for $I^i_{\infty}$, and readers are encouraged to view $I^i_m$ as `$I^i/I^m$'.

For $j\leq i$ and $n \leq m$, it is not hard to see that the evident map
$$ \Oo_m^i \ra \Oo_n^j$$
is a map of $\Oo$--bimodules, and thus induces a natural maps
$$ I_m^i \ra I^j_n$$
for all $I \in \AlgO(R)^c$.

Special cases of these are illustrated in the next examples.

\begin{ex}  $I \in \AlgO(R)^c$ has a natural `augmentation ideal' filtration
$$ I \xla{\sim} I^1 \la I^2 \la I^3 \la \dots.$$
\end{ex}

\begin{ex}  $I^1_n = P_{n-1}F^R_{\Oo}(I)$ in the notation of the last section, so the tower
$$ I^1_2 \la I^1_3 \la I^1_4 \la \dots$$
identifies with the Goodwillie tower of the identity functor on $\AlgO(R)$.  This tower, defined as we do here, is the subject of the study \cite{HH}.
\end{ex}

These examples are related: the filtration of the first example appears as the homotopy fibers of the maps from $I$ to the tower in the second example.  More precisely, there are homotopy fiber sequences
$$ I^n \ra I^1 \ra I^1_n.$$
This is a special case of property (b) in the next theorem.

\begin{thm} \label{slices thm}  The functors $I \mapsto I^i_n$ satisfy the following properties.

\vskip 8pt

\noindent{\bf (a)} They preserve weak equivalences in the variable $I \in \AlgO(R)^c$.

\vskip 8pt

\noindent{\bf (b)} For $k<l<m$, the sequence $ I^l_m \ra I^k_m \ra I^k_l$
is a homotopy fiber sequence.

\vskip 8pt

\noindent{\bf (c)} There are natural isomorphisms $I^1_2 = z(TQ(I))$, and more generally, $I_{k+1}^k = z(\Oo(k) \sm_{\Sigma_k\wr \Oo(1)} TQ(I)^{\sm_R k})$.

\vskip 8pt

\noindent{\bf (d)} Let $R \ra R^{\prime}$ be a map of commutative $S$--algebras.  There is a natural isomorphism $R^{\prime} \sm_R I^i_n \simeq (R^{\prime} \sm_R I)^i_n$.
\end{thm}

All of these properties follow immediately from the more general results of \secref{GEN RESULTS SEC}.  For example, part (b) follows from \thmref{NAT IN M THM}(b) applied to the sequence of $\Oo$--bimodules
$$ \Oo_l^m \ra \Oo_k^m \ra \Oo_k^l.$$

Our connectivity estimates of \secref{goodwillie tower sec} give the following.

\begin{prop} \label{In connectivity prop}  Suppose $R$ and $\Oo$ are connective. If $I$ is $(c-1)$--con\-nected, then $I^n$ is $(nc-1)$--connected.
\end{prop}

\subsection{Composition properties of the filtration}

Now we look at composition structure.  It is not hard to see that the operad composition
$$ \mu: \Oo \circ \Oo \ra \Oo$$
induces maps of $\Oo$--bimodules
$$ \mu: \Oo_i^{\infty} \co \Oo_j^{\infty} \ra \Oo_{ij}^{\infty}.$$

These pairings in turn define natural maps
$$ \mu: (I^j)^i \ra I^{i j}$$
for all $I \in \AlgO(R)^c$.

\vskip 8pt

With a little more care, one can check the following.

\begin{lem} $\mu: \Oo \circ \Oo \ra \Oo$ induces maps of $\Oo$--bimodules
$$ \mu: \Oo_i^m \co \Oo_j^n \ra \Oo_{i j}^{\min(i j+(n-j),m j)},$$
and thus induces natural maps
$$ \mu: (I^j_n)^i_m \ra I^{i j}_{\min(i j+(n-j),m j)}.$$
\end{lem}

\begin{proof}

We first check that if $N = \min(ij+(n-j),mj)$, then the dotted arrow exists in the diagram
\begin{equation*}
\xymatrix{
\Oo_i^{\infty} \circ \Oo_j^{\infty} \ar[d] \ar[r] & \Oo_{i j}^{\infty} \ar[d]  \\
\Oo_i^{m} \circ \Oo_j^{n}           \ar@{.>}[r]   & \Oo_{i j}^{N}. }
\end{equation*}

Now $(\Oo_i^{\infty} \circ \Oo_j^{\infty})(s)$ equals the wedge of $S$--modules of the form $\Oo(r) \sm \Oo(s_1) \sm \dots \sm \Oo(s_r)$ such that $s=s_1+\dots+s_r$, $i \leq r$, and $j\leq s_k$ for all $k$.  (All such modules occur, some with multiplicities greater than 1.)

Such a wedge summand maps to $*$ under the quotient $\Oo_i^{\infty} \circ \Oo_j^{\infty} \ra \Oo_i^{m} \circ \Oo_j^{n}$ if either $r\geq m$ or $s_k \geq n$ for at least one $k$.  In the first case, it follows that $s\geq mj$. In the second case, it follows that $s\geq (r-1)j+n \geq (i-1)j+n = i j + (n-j)$.  We conclude that if $N$ is less than or equal to both $mj$ and $ij+(n-j)$, then the dotted arrow in the diagram above exists.

The bimodule map $\Oo_i^m \circ \Oo_j^n \ra \Oo_{i j}^{\min(i j+(n-j),m j)}$ then induces an $\Oo$--bimodule map
$\Oo_i^m \co \Oo_j^n \ra \Oo_{i j}^{\min(i j+(n-j),m j)}$.  This follows formally from the fact that each of the maps $\Oo \hookleftarrow \Oo_i^{\infty} \twoheadrightarrow \Oo_i^m$ are maps of $\Oo$--bimodules, combined with the evident fact that the operad pairing $\Oo \circ \Oo \ra \Oo$ induces a map $\Oo \co \Oo \ra \Oo$.
\end{proof}

\begin{ex} For simplicity, let $D_i(M) = \Oo(i) \sm_{\Sigma_i \wr \Oo(1)} M^{\sm_R i}$, for $M \in R\Oo(1)\modd$, and let $T=TQ$.  With this notation, there is an isomorphism $I^i_{i+1} \simeq zD_iT(I)$, and a commutative diagram
\begin{equation*}
\xymatrix{
(I^j_{j+1})^i_{i+1} \ar@{=}[d] \ar[rr]^{\mu} && I^{i j}_{i j+1} \ar@{=}[d]  \\
z D_i T(z D_j T(I)) \ar[r] &z D_i D_j T(I) \ar[r] &z D_{i j}T(I) }
\end{equation*}
where the lower left map is induced by the counit $T z M \ra M$, and the lower right map is induced by the operad structure map $\Oo(i) \sm \Oo(j)^{\sm i} \ra \Oo(i j)$.
\end{ex}

\subsection{Application to lifting filtrations}

\begin{thm} \label{lifting thm}  Let $I,J \in \AlgO(R)^c$, and let $f: I \ra J^d$ be a morphism in $\AlgO(R)$.  Then $f$ induces compatible $\Oo$--algebra maps $f_n: I^n \ra J^{dn}$ for all $n$, and the assignment $f \mapsto f_n$ is both functorial and preserves weak equivalences.
\end{thm}
\begin{proof}  Let $f_n$ be the composite $I^n \xra{f^n} (J^d)^n \xra{\mu} J^{dn}$.
\end{proof}

\begin{defn}  \label{AQ FILT DEFN} Say that a map $f \in [I,J]_{\Alg}$ has AQ-filtration\footnote{The reader can decide if AQ stands for Andr\'e-Quillen or Adams-Quillen.} $s$ if $f$ factors in $ho(\AlgO(R))$ as the composition of $s$ maps
$$I = I(0) \xra{f(1)} I(1) \xra{f(2)} I(2) \ra \dots \ra I(s-1) \xra{f(s)} I(s)=J$$
such that $TQ(f(i))$ is null for each $i$..
\end{defn}

\begin{thm} \label{AQ lifting thm} Let $f \in [I,J]_{\Alg}$ have AQ-filtration $s$.  Then there exists $\tilde f \in [I, J^{2^s}]_{\Alg}$ such that
\begin{equation*}
\xymatrix{
& J^{2^s} \ar[d]  \\
I \ar[r]^-f \ar[ru]^-{\tilde f}& J }
\end{equation*}
commutes in $ho(\AlgO(R))$.
\end{thm}

\begin{proof} We work in $ho(\AlgO(R))$.

Let $f = f(s) \circ \dots \circ f(1)$ as in Definition \ref{AQ FILT DEFN}.

For each $i$ between 1 and $s$, there is an exact sequence of pointed sets
$$ [I(i-1),I(i)^2]_{\Alg} \ra [I(i-1),I(i)]_{\Alg} \ra [I(i-1),I(i)^1_2]_{\Alg},$$
and there are identifications
$$ [I(i-1),I(i)^1_2]_{\Alg} \simeq [I(i-1),z(TQ(I(i)))]_{\Alg} \simeq [TQ(I(i-1)),TQ(I(i))]_{\Mod}.$$

It follows that since $TQ(f(i))$ is null, $f(i)$ lifts to $\tilde f(i): I(i-1) \ra I(i)^2$.  \thmref{lifting thm} then gives maps $\tilde f(i)_{2^{i-1}}: I(i-1)^{2^{i-1}} \ra I(i)^{2^i}$.  Now let $\tilde f$ be the composite of these $s$ maps:
$ \tilde f = \tilde f(s)_{2^{s-1}} \circ \dots \circ \tilde f(1)$.
\end{proof}

The theorem, combined with \propref{In connectivity prop}, has the following corollary.

\begin{cor} \label{AQ lifting cor}  Suppose that $R$ and $\Oo$ are connective and $J \in \Alg_{\Oo}(R)$ is $(c-1)$--connected.  If $f: I \ra J$ has $AQ$-filtration $s$, then $f_*:\pi_*(I) \ra \pi_*(J)$ will be zero for $*<2^sc$.
\end{cor}

For more results in this spirit see \cite{kuhn hurewicz}.

\section{Deferred proofs} \label{deferred proofs sec}

In this section we prove Propositions \ref{COFFGT PROP}, \ref{BARCOF PROP}, and \ref{BARDEREQ PROP} and Theorem \ref{NAT IN M THM}. When $R=S$, so that our algebras just have the underlying structure of an $S$--module, these results can be deduced from the second author's work, specifically  \cite[Thm.1.1]{Pereira2014}. The case of a general $R$ requires a suitable generalization of that result, which we state as Theorem \ref{CIRCO POS R THM} below.

\subsection{The homotopical behavior of the composition product}

Fixing a commutative $S$--algebra $R$, it is useful to generalize the context slightly.

\begin{notation}
Let $\mathcal{P}$ be an operad in $R\modd$, i.e. a monoid object for the monoidal structure $\circ_R$ in $\mathsf{Sym}(R)$ defined just as in (\ref{COMPPROD EQ}) but with $\wedge$ replaced by $\wedge_R$.   We then denote by  $\Mod^r_{\mathcal{P}}$, $\Mod^l_{\mathcal{P}}$, and $\Alg_{\mathcal{P}}$ the associated categories of left modules, right modules, and algebras over $\mathcal{P}$ in $\Sym(R)$.  We endow $\Mod^l_{\mathcal{P}}$, and $\Alg_{\mathcal{P}}$ with the model structure as in \thmref{PROJPOS THM}\footnote{Note that we do not equip $\Mod^r_{\mathcal{P}}$ with a model structure.}.
\end{notation}

\begin{rem}\label{INDUCEOPERAD REM}
If $\Oo$ is an operad in $S\modd$, then the symmetric sequence $R\wedge \Oo$, defined as $(R \wedge \Oo)(n)=R\wedge \Oo(n)$ is naturally an operad in $R\modd$.  It can be readily checked that there are isomorphisms of model categories
\begin{equation}\label{IDENTIFICATION EQ}
\Alg_{R \wedge \Oo} \simeq \Alg_{\Oo}(R) \text{ and } \Mod^l_{R \wedge \Oo} \simeq \Mod^l_{\Oo}(R).
\end{equation}
\end{rem}

To state our main technical theorem, we need the following construction.

\begin{defn}
Suppose given a map $f_1: M \ra N$ in $\Mod^r_{\mathcal{P}}$
 and a map $f_2: A \ra B$ in $\Mod^l_{\mathcal{P}}$.
Let $(M \circ_{\mathcal{P}} B) \vee_{M \circ_{\mathcal{P}} A} (N \circ_{\mathcal{P}} A)$ be the pushout of the diagram
 \begin{equation*}
\xymatrix{
M \circ_{\mathcal{P}} A \ar[d]_{M \circ_{\mathcal{P}} f_2 } \ar[r]^{f_1 \circ_{\mathcal{P}} A} & N \circ_{\mathcal{P}} A  \\
M \circ_{\mathcal{P}} B                                                                        &  }
\end{equation*}
in $\Sym(R)$,
and then define the {\em pushout circle product} of $f_1$ and $f_2$, to be the natural map
$$ f_1 \square^{\circ_{\mathcal{P}}} f_2: (M \circ_{\mathcal{P}} B) \vee_{M \circ_{\mathcal{P}} A} (N \circ_{\mathcal{P}} A) \ra N \circ_{\mathcal{P}} B.$$
\end{defn}

\begin{thm}\label{CIRCO POS R THM} Suppose $f_2: A \ra B$ is a cofibration between cofibrant objects in $\Mod^l_{\mathcal{P}}$. If a $f_1: M \ra N$ in $\Mod^r_{\mathcal{P}}$ is an underlying positive cofibration in $\Sym(R)$, then so is
$$ f_1 \square^{\circ_{\mathcal{P}}} f_2: (M \circ_{\mathcal{P}} B) \vee_{M \circ_{\mathcal{P}} A} (N \circ_{\mathcal{P}} A) \ra N \circ_{\mathcal{P}} B.$$
Furthermore, this map will be a weak equivalence if either $f_1$ or $f_2$ is a weak equivalence.
\end{thm}

When $R=S$, this theorem nearly coincides with \cite[Thm.1.1]{Pereira2014}, and we defer the proof in the general case to the appendix.  For the purpose of proving results stated in \secref{GEN RESULTS SEC}, we will just need the following corollary.

\begin{cor}\label{CIRCO POS THM} Let $\Oo$ be an operad in $S\modd$.
Suppose $f_2: I \ra J$ is a cofibration between cofibrant objects in $\Alg_{\Oo}(R)$. If a map $f_1: M \ra N$ in $\Mod^r_{\Oo}(S)$ is an underlying positive cofibration in $\Sym(S)$, then
$$ f_1 \square^{\circ_{\mathcal{O}}} f_2: (M \circ_{\mathcal{P}} J) \vee_{M \circ_{\mathcal{O}} I} (N \circ_{\mathcal{O}} I) \ra N \circ_{\mathcal{O}} J.$$
will be a positive cofibration in $R\modd$.

Furthermore, this map will be a weak equivalence if either $f_1$ or $f_2$ is a weak equivalence.
\end{cor}

\begin{proof}
Since the functor $R\wedge \underline{\text{\hspace{.13in}}}\colon \Sym(S)\to \Sym(R)$ sends positive cofibrations and  trivial cofibrations in $\Sym(S)$ respectively to positive cofibrations and  trivial cofibrations in $\Sym(R)$, the result follows immediately from Theorem \ref{CIRCO POS R THM} applied to $\mathcal P = R \sm \Oo$, $R \sm f_1$ and $f_2$.  Note that the positive model structure on $\Sym(R)$ restricts on level 0 to the positive model structure on $R\modd$.
\end{proof}

\begin{rem}  Trying to directly prove the corollary -- the result we need for this paper -- has led us to the more general  Theorem \ref{CIRCO POS R THM}, and, in particular, the use of operads $\mathcal P$ more general than $R \sm \Oo$.  This generality comes at a price: for a general $R$, the positive model structure on $\Sym(R)$ is more subtle than the positive model structure on $\Sym(S)$.
\end{rem}

\begin{rem} It is straightforward to verify that the operads $R \wedge \Oo$ can also be regarded as operads in $S\modd$, so that (\ref{IDENTIFICATION EQ}) extends to give
\begin{equation}
\Alg_{R \wedge \Oo} \simeq \Alg_{\Oo}(R) \simeq \Alg_{R \wedge \Oo}(S)
\label{IDENTIFICATION2 EQ}
\end{equation}
and
\begin{equation}
\Mod^l_{R \wedge \Oo} \simeq \Mod^l_{\Oo}(R) \simeq \Mod^l_{R \wedge \Oo}(S).
\label{IDENTIFICATION2.5 EQ}
\end{equation}

A slightly more careful analysis shows one also has an inclusion of categories
\begin{equation}\label{IDENTIFICATION3 EQ}
\Mod^r_{R \wedge \Oo} \subset \Mod^r_{R \wedge \Oo}(S).
\end{equation}
It might be surprising at first that this is not an isomorphism. Unwinding definitions, one sees that given $N \in \Mod^r_{R \wedge \Oo}(S)$, $N(n)$ will be a $\Sigma_n \wr R$--module. For any $N$ coming from $\Mod^r_{R \wedge \Oo}$, this $\Sigma_n \wr R$--module structure must be one pulled back along the canonical ring map $\Sigma_n \wr R \to R$, which is not the case in general.
\end{rem}

\subsection{Proofs of results from \secref{GEN RESULTS SEC}} \label{proofs subsection}

\begin{proof}[Proof of Proposition \ref{COFFGT PROP}]

If $f_1$ is the map $* \ra \Oo$, and $f_2: I \ra J$ is map in $\Alg_{\Oo}(R)$, then $f_1 \square^{\circ_{\Oo}} f_2$ is just the map $f_2: I \ra J$, now viewed as a map in $R\modd$.

If $I$ is cofibrant in $\Alg_{\Oo}(R)$, then applying Corollary \ref{CIRCO POS THM} to the map $f_2: * \ra I$, shows that $I$ will be cofibrant in $R\modd$.

Similarly, if $f_2: I \ra J$ is a cofibration between cofibrant objects in $\Alg_{\Oo}(R)$,  we learn that  $f_2:I \to J$ is a cofibration in $R\modd$.
\end{proof}

\begin{proof}[Proof of Proposition \ref{BARCOF PROP}]

For the first statement, we note that $B(M,\Oo,N)$ is the realization of the simplicial object $B_{\bullet}(M,\Oo,N)$, and thus will be cofibrant in $\Sym(S)$ if $B_{\bullet}(M,\Oo,N)$ is Reedy cofibrant in $\Sym(S)^{\Delta^{op}}$.  That this is true, under our hypotheses on $M$ and $N$, is precisely the conclusion of \cite[Thm. 1.6]{Pereira2014}.

Proving the second statement is similar: one sees that $B_{\bullet}(M,\Oo,I)$ is Reedy cofibrant in $R\modd^{\Delta^{op}}$ by noting that the proof of \cite[Thm.1.6]{Pereira2014} (and in particular that of the auxiliary \cite[Lem.5.47]{Pereira2014}) goes through if one simply replaces the very last application of \cite[Thm.1.1]{Pereira2014} by an application of Corollary \ref{CIRCO POS THM}.
\end{proof}

\begin{proof}[Proof of Proposition \ref{BARDEREQ PROP}]
First note that by Corollary \ref{CIRCO POS THM} the functor
\[M\circ_{\Oo} \underline{\text{\hspace{.13in}}}\colon \Alg_{\Oo}(R) \to R\modd\]
sends trivial cofibrations between cofibrant algebras to weak equivalences, and hence, by Ken Brown's lemma \cite[Cor.7.7.2]{hirschhorn}, also preserves all weak equivalences between cofibrant algebras.

Hence, rewriting the map
\[B(M,\Oo,I)\to M \circ_{\Oo}I \]
as
\[M\circ_{\Oo}(B(\Oo,\Oo,I)\to I)\]
one sees it suffices to show that $B(\Oo,\Oo,I)$ is cofibrant in $\Alg_{\Oo}(R)$.

$B(\Oo,\Oo,I)$ is the realization of the simplicial algebra $B_{\bullet}(\Oo,\Oo,I)$, viewed as a simplicial object in $R\modd$.  By \cite[Prop.6.11]{HH}, this agrees with the realization of $B_{\bullet}(\Oo,\Oo,I)$, viewed as a simplicial object in $\Alg_{\Oo}(R)$.  Thus it suffices to show that $B_{\bullet}(\Oo,\Oo,I)$ is Reedy cofibrant in $\Alg_{\Oo}(R)^{\Delta^{op}}$.

Checking this involves showing that the latching maps for  $B_{\bullet}(\Oo,\Oo,I)$ are cofibrations in $\Alg_{\Oo}(R)$.  These depend only on $B_{\bullet}(\Oo,\Oo,I)$ together with its degeneracies, i.e. face maps can be ignored.  From this perspective,
$$ B_{\bullet}(\Oo,\Oo,I) \simeq  \Oo \circ B_{\bullet}(S(1),\Oo,I),$$
where $S(1)$ is our notation for the unit symmetric sequence $(*,S,*,*,\dots)$ under $\circ$.
Hence, letting $\ell_n^{\Oo}$ and $\ell_n$ respectively denote the $n$th latching map construction on $\mathbb N$--graded objects with degeneracies in $\AlgO(R)$ and $R\modd$, one has
\[\ell_n^{\Oo}\left(B_{\bullet}(\Oo,\Oo,I)\right)\simeq \Oo \circ \ell_n\left( B_{\bullet}(S(1),\Oo,I). \right)\]

Since $\Oo \circ \underline{\text{\hspace{.13in}}}: R\modd \to \AlgO(R)$ is a left Quillen functor, it follows that $\ell_n^{\Oo}\left(B_{\bullet}(\Oo,\Oo,I)\right)$  will be a cofibration in  $\Alg_{\Oo}(R)$ if $\ell_n\left(B_{\bullet}(S(1),\Oo,I)\right)$ is a cofibration in $R\modd$.  But the latter map {\em is} a cofibration, since it is a special case of the latching maps shown to be cofibrations in the proof of Proposition \ref{BARCOF PROP}.
\end{proof}

\begin{proof}[Proof of Theorem \ref{NAT IN M THM} (a) and (b)]

In this proof we focus on the identification
$\AlgO(R) \simeq \mathsf{Alg}_{R\wedge \Oo}(S)$
so as to be able to directly apply \cite[Thm.1.1]{Pereira2014}.

For part (a), note first that
\[F^R_M(I)=M\circ_{\Oo}B(\Oo,\Oo,I).\]
That $F^R_M(I)$ preserves weak equivalences in the $I$ variable then follows from the proof of Proposition \ref{BARDEREQ PROP}, where it was shown both that $B(\Oo,\Oo,I)$ is a cofibrant algebra and that $M\circ_{\Oo} \underline{\text{\hspace{.13in}}}$ preserves w.e.s between cofibrant algebras.

To see that weak equivalences are also preserved in the $M$ variable, one uses a similar argument: using the identifications (\ref{IDENTIFICATION2 EQ}) and (\ref{IDENTIFICATION3 EQ}) to change perspective to $S\modd$, one applies \cite[Thm.1.1]{Pereira2014} to any trivial cofibration $f_1\colon M \to N$
in $\Mod^r_{R \wedge \Oo}(S)$
and the map $f_2=\** \to B(\Oo,\Oo,I)$.
One concludes that the functor $M\mapsto F^R_M(I)$ sends trivial cofibrations to weak equivalences.  One now again uses Ken Brown's lemma.

The intuition behind part (b) comes from the observation that (\ref{COMPPROD EQ}), the formula for the composition product $X \circ Y$ of symmetric sequences,  is `linear' in the variable $X$. Our official proof goes as follows.  Note that homotopy fibration sequences in $\Alg_{\Oo}(R)$ are detected by considering them as sequences in $S\modd$.  Again using the identifications (\ref{IDENTIFICATION2 EQ}) and (\ref{IDENTIFICATION3 EQ}) to change perspectives, one immediately reduces to \cite[Thm. 1.8]{Pereira2014} applied to the the operad $R \wedge \Oo$ in $S$--modules.
\end{proof}

\appendix

\section{Proof of Theorem \ref{CIRCO POS R THM}}\label{app a}

We now turn to the task of proving Theorem \ref{CIRCO POS R THM}. If one just tries to redo all the work in \cite{Pereira2014} with a general commutative $S$--algebra $R$ replacing occurrences of $S$, one finds that most of results generalize, with the key exception being the characterization of $S$ cofibrations in \cite[Prop. 3.9]{Pereira2014}, which fails for general $R$. Here we take a somewhat blended approach: we use a string of arguments from \cite{Pereira2014} to ultimate reduce ourselves to a situation covered by  \cite[Thm.1.1]{Pereira2014}.

We begin by noting the following elementary lemma, a consequence of the fact that the positive model structure on $\Sym(R)$ is the projective structure induced from the positive model structure on $\Sym(S)$.

\begin{lem} \label{GEN COFIB LEMMA} A set of generating cofibrations for $\Sym(R)$ can be obtained by applying $R \sm \underline{\text{\hspace{.13in}}}$ to a set of generating cofibrations for $\Sym(S)$.
\end{lem}

Let us remind ourselves of our goal.  Given $f_1: M \ra N$ in $\Mod^r_{\mathcal{P}}$ and $f_2: A \ra B$ in $\Mod^l_{\mathcal{P}}$, we are considering the `pushout corner map', in $\Sym(R)$, of the commutative square
\begin{equation} \label{basic square eq}
\xymatrix{
M \circ_{\mathcal{P}} A \ar[r]^{M \circ_{\mathcal{P}} f_2 } \ar[d]^{f_1 \circ_{\mathcal{P}} A} & M \circ_{\mathcal{P}} B  \ar[d]^{f_1 \circ_{\mathcal{P}} B}\\
N \circ_{\mathcal{P}} A   \ar[r]^{N \circ_{\mathcal{P}} f_2 }                                   &  N \circ_{\mathcal{P}} B.}
\end{equation}
By this we mean the map from the pushout of the upper left corner of the square to the lower right term.

We wish to show that if $f_2$ is a cofibration between cofibrant objects\footnote{This is a bit redundant: if $A$ is cofibrant, and $f_2$ is a cofibration, then $B$ is necessarily cofibrant.} in $\Mod^l_{\mathcal{P}}$, then if $f_1$ is a positive cofibration in $\Sym(R)$, so is the pushout corner map.  Furthermore, in this situation, if either $f_1$ or $f_2$ is a weak equivalence, so is the pushout corner map.

We will address this last statement at the end of the appendix, and focus on the first statement.  For this, we try to follow the proof of \cite[Thm. 1.1]{Pereira2014}, which is the case when $R=S$. The majority of the arguments in that proof are agnostic as to the category or model structure used -- in particular, the filtrations of \cite[Prop. 5.20]{Pereira2014} cover $R\modd$ -- with the exception of the two instances where \cite[Thms 1.2, 1.3]{Pereira2014} are used.

As in \cite{Pereira2014}, we first assume that $f_2$ is a map between algebras, rather than more general left $\mathcal P$--modules.  In this case, arguing as in \cite[\S 5.4]{Pereira2014}, one reduces to the case where $f_2\colon A \to B$ is the pushout of a generating cofibration. Using \lemref{GEN COFIB LEMMA}, this means that $f_2$ is the lower horizontal map of a pushout in $\Alg_{\mathcal P}$ of the form
\begin{equation*}
\xymatrix{
\mathcal{P} \circ_R (R \wedge X) \ar[d] \ar[r] & \mathcal{P} \circ_R (R \wedge Y) \ar[d]  \\
A \ar[r]^{f_2} & B }
\end{equation*}
with $X \ra Y$ a generating positive cofibration in $S \modd$.

The key is now to use the infinite filtration of the horizontal maps in (\ref{basic square eq}) given by \cite[Prop. 5.20]{Pereira2014}. (This key filtration is a generalization of similar filtrations appearing in \cite[Prop. 5.10]{HH} and \cite[Proofs of Thms.~ 1.4 and 12.5]{elmendorf mandell}.)  Arguing as in \cite[\S 5.4]{Pereira2014}, one is reduced to studying the pushout corner maps of the squares
\begin{equation}\label{MA PUSH EQ}
 \xymatrix{
     M_A(r) \wedge_{R \times \Sigma_r} (R \wedge Q^r_{r-1}) \ar[r] \ar[d] &
       M_A(r) \wedge_{R \times \Sigma_r} (R \wedge Y^{\wedge r}) \ar[d]
    \\
       N_A(r) \wedge_{R \times \Sigma_r} (R \wedge Q^r_{r-1}) \ar[r] &
         N_A(r) \wedge_{R \times \Sigma_r} (R \wedge Y^{\wedge r}),
 }
\end{equation}
where we need to explain our notation.

Firstly, if we view $X \ra Y$ as a functor $\{ \mathbf{0} \ra \mathbf{1}\} \ra S\modd$, we can smash this functor with itself $r$ times, obtaining a cubical diagram
$\{ \mathbf{0} \ra \mathbf{1}\}^{\times r} \ra S\modd$.  We let $Q^r_{r-1}$ denote the colimit of this cube with the terminal object $\mathbf{1}^r$ removed; this comes with an evident map $Q^r_{r-1} \ra Y^{\sm r}$.

Secondly, as in \cite[Definition 5.15]{Pereira2014}, $M_A$ denotes the $M \circ_{\mathcal P} (\mathcal P \coprod A)$, where the coproduct is taken in $\Mod^l_{\mathcal{P}}$.

We wish to show that the pushout corner map of (\ref{MA PUSH EQ}) is a positive cofibration in  $R \modd$.  Since $X \ra Y$ is a positive cofibration in $S \modd$, \cite[Theorem 1.2]{Pereira2014} tells us that $Q^r_{r-1} \ra Y^{\sm r}$ is appropriately cofibrant in the category of $S$--modules with a $\Sigma_r$ action.

If the map $M_A\to N_A$ were a generating positive cofibration in $\Sym(R)$, one would be able to pull a $R \wedge (-)$ factor out of the pushout corner map (by Lemma \ref{GEN COFIB LEMMA}), reducing to the $S$ case which in turn follows by applying \cite[Thms 1.2 and 1.3]{Pereira2014} as in the proof of \cite[Thm. 1.1]{Pereira2014}.

Hence, by standard arguments, it suffices to show that $M_A \to N_A$ is a positive cofibration in $\Sym(R)$.  This would follow from the the special case of our theorem when $f_2$ has the form $i: \mathcal P \ra \mathcal P \coprod A$, which would say that the pushout corner map of the middle square of the diagram
\begin{equation} \label{2nd basic square eq}
\xymatrix{
M \ar@{=}[r] \ar[d]^{f_1}& M \circ_{\mathcal{P}} \mathcal P \ar[rr]^{M \circ_{\mathcal{P}} i } \ar[d]^{f_1 \circ_{\mathcal{P}} \mathcal P} && M \circ_{\mathcal{P}} (\mathcal P \coprod A)  \ar[d]^{f_1 \circ_{\mathcal{P}} (\mathcal P \coprod A)} \ar@{=}[r] & M_A \ar[d]\\
N \ar@{=}[r] & N \circ_{\mathcal{P}} \mathcal P   \ar[rr]^{N \circ_{\mathcal{P}} i }                                   &&   N \circ_{\mathcal{P}} (\mathcal P \coprod A) \ar@{=}[r] & N_A}
\end{equation}
is a positive cofibration in $\Sym(R)$.

Now we use our assumption that $A$ is cofibrant in $\Alg_{\mathcal P}$, and basically proceed as before.  The map $i$ can be assumed to be an infinite composition of maps of the form $\mathcal P \coprod A_{\beta} \xra{\mathcal P \coprod i_{\beta}} \mathcal P \coprod A_{\beta+1}$, where $i_{\beta}$ is the lower horizontal map of a pushout in $\Alg_{\mathcal P}$ of the form
\begin{equation*}
\xymatrix{
\mathcal{P} \circ_R (R \wedge X_{\beta}) \ar[d] \ar[r] & \mathcal{P} \circ_R (R \wedge Y_{\beta}) \ar[d]  \\
A_{\beta} \ar[r] & A_{\beta + 1} }
\end{equation*}
with $X_{\beta} \ra Y_{\beta}$ a generating positive cofibration in $S \modd$.

It suffices to show by induction on $\beta$ that
$N_{A_{\beta}} \coprod_{M_{A_{\beta}}} M_{A_{\beta+1}} \to N_{A_{\beta+1}}$
is a positive cofibration. Note that the induction hypothesis then implies $M_{A_{\beta}} \ra N_{A_{\beta}}$ is a positive cofibration.

After a filtration argument as before, one is left needing to show that the pushout corner map in
\[
 \xymatrix{
     M_{(\mathcal{P} \coprod A_{\beta})}(r) \check{\wedge}_{R \times \Sigma_r} (R \wedge Q^r_{r-1}) \ar[r] \ar[d] &
     M_{(\mathcal{P} \coprod A_{\beta})}(r) \check{\wedge}_{R \times \Sigma_r} (R \wedge Y_{\beta}^{\check{\wedge} r} \ar[d])
    \\
     N_{(\mathcal{P} \coprod A_{\beta})}(r) \check{\wedge}_{R \times \Sigma_r} (R \wedge Q^r_{r-1}) \ar[r] &
     N_{(\mathcal{P} \coprod A_{\beta})}(r) \check{\wedge}_{R \times \Sigma_r} (R \wedge Y_{\beta}^{\check{\wedge} r}).
 }
\]
is a positive cofibration in $\Sym(R)$, where $\check{\wedge}$ denotes the smash product in $\Sym(R)$.

Using the obvious analogue of \lemref{GEN COFIB LEMMA} for $R$ bi-symmetric sequences and \cite[Props 5.43, 5.44]{Pereira2014} (the analogues of \cite[Thms 1.2,1.3]{Pereira2014} for $\mathsf{Sym}(S)$) just as in the argument following (\ref{MA PUSH EQ}), one further reduces to just needing to show that $M_{\mathcal{P} \coprod A_{\beta}}\to N_{\mathcal{P} \coprod A_{\beta}}$  is a positive cofibration in $\biSym(R)$, the category of bi--symmetric sequences of $R$--modules. (The notion of cofibration is defined by analogy with $\Sym(R)$). But since \cite[Prop. 5.19]{Pereira2014} identifies the $(r,s)$ level of this map with $M_{A_{\beta}}(r+s) \to N_{A_{\beta}}(r+s)$, the result follows by our induction hypothesis.

To deal with the case of $f_2$ a general map of left modules one repeats the argument in the last paragraph of the proof of \cite[Thm. 1.1]{Pereira2014}.

Finally, the case where either $f_1$ or $f_2$ are additionally weak equivalences follows by using identifications (\ref{IDENTIFICATION2 EQ}) and (\ref{IDENTIFICATION3 EQ}) to reduce the question to the $S\modd$ level and then applying the monomorphism part of \cite[Thm. 1.1]{Pereira2014}.


\begin{thebibliography}{99}

\bibitem[B]{basterra} M.~Basterra, {\em Andr\'e--Qullen cohomology of commutative $S$--algebras}, J. Pure Appl. Alg. {\bf 144}(1999), 111-143.



\bibitem[EM]{elmendorf mandell} A.~D.~Elmendorf and M.~A.~Mandell, {\em Rings, modules, and algebras in infinite loop space theory}, Adv. Math. {\bf 205} (2006), 163--228.


\bibitem[H1]{H07} J.~E.~Harper, {\em Homotopy theory of modules over operads in symmetric spectra}, Algebr. Geom. Topol., {\bf 9}(2009), 1637-1680.

\bibitem[H2]{H09} J.E.Harper, {\em Bar constructions and Quillen homology of modules over operads}, Alg. Geom. Top. {\bf 10} (2010), 87--136.

\bibitem[HH]{HH} J.E.Harper and K.Hess, {\em Homotopy completion and topological Quillen homology of structured ring spectra}, Geom. Topol. {\bf 17} (2013), 1325–1416.

\bibitem[Hir]{hirschhorn} P.S.Hirschhorn, {\em  Model categories and their localizations}, A.M.S. Math. Surveys and Monographs 99, 2003.

\bibitem[HSS]{hss} Mark Hovey, Brooke Shipley, and Jeff Smith, {\em Symmetric spectra}, J. Amer. Math. Soc. {\bf 13} (2000), 149–208.

\bibitem[K1]{kuhn-TAQtowers} N. ~J. ~Kuhn, {\em Localization of Andr\'{e}--Quillen--Goodwillie towers, and the periodic homology of infinite loopspaces}, Adv. Math. {\bf 201} (2006), 318--378.


\bibitem[K2]{kuhn hurewicz} N.~J.~Kuhn, {\em Adams filtration and generalized Hurewicz maps for infinite loopspaces}, preprint, arXiv:1403.7501.

\bibitem[MMSS]{mmss} M.~A.~ Mandell, J.~ P.~ May, S.~ Schwede, and B.~ Shipley, {\em Model categories of diagram spectra}, Proc. London Math. Soc. {\bf 82} (2001), 441–512.


\bibitem[P1]{Pereira2013} Lu\'is A. Pereira, \emph{Goodwillie calculus in the category of operads over a spectral operad}, preprint, October, 2015.

\bibitem[P2]{Pereira2014} Lu\'is A. Pereira, \emph{Cofibrancy of Operadic Constructions in positive
symmetric spectra}, Homology, Homotopy, and Applications, to appear.

\bibitem[S]{shipley} B.~ Shipley, {\em A convenient category for commutative ring spectra}, A.M.S. Cont. Math. {\bf 346} (2004), 473--483.

\end{thebibliography}
\end{document}